\documentclass[11pt]{amsart}

\usepackage{amssymb, amsmath, amsthm, mathtools, floatflt, palatino, euler, epic, eepic, microtype}
\usepackage{comment}
\usepackage{cases}
\usepackage{graphicx}
\usepackage[colorlinks, linkcolor=darkred, citecolor=darkblue, pagebackref=true, pdftex]{hyperref}
\usepackage[linesnumbered,ruled]{algorithm2e} 
\usepackage[normalem]{ulem}
\usepackage[usenames]{xcolor}
\usepackage[margin=1 in]{geometry}
\usepackage{subcaption}

\makeatletter
\def\BState{\State\hskip-\ALG@thistlm}
\makeatother

\definecolor{darkblue}{rgb}{0,0,0.7}
\definecolor{darkred}{rgb}{0.7,0,0}

\newcommand{\defi}[1]{\textit{\color{blue}#1}}    

\newtheorem{theorem}{Theorem}
\newtheorem{lemma}[theorem]{Lemma}
\newtheorem{axiom}[theorem]{Statement}
\newtheorem{corollary}[theorem]{Corollary}
\newtheorem{definition}[theorem]{Definition}
\newtheorem{proposition}[theorem]{Proposition}
\newtheorem{conjecture}[theorem]{Conjecture}
\newtheorem{remark}[theorem]{Remark}
\newtheorem{question}[theorem]{Question}
\newtheorem*{theorem*}{Theorem}
\newtheorem*{cor*}{Corollary}
\newtheorem*{prop*}{Proposition}
\newtheorem*{conjecture*}{Conjecture}

\newcommand{\II}{\mathfrak{I}}
\newcommand{\SIC}{\mathfrak{I}_S}
\newcommand{\SVIC}{\mathfrak{I}_{SV}}
\newcommand{\VIC}{\mathfrak{I}_{V}}

\begin{document}

\title[Minimal graphs for contractible and dismantlable properties]{Minimal graphs for contractible and dismantlable properties}

\author[A. Dochtermann]{Anton Dochtermann}
\email{dochtermann@txstate.edu}
 \address{Department of Mathematics, Texas State University, USA}
 
\author[J. F. Espinoza]{Jes\'us F. Espinoza}
\email{jesusfrancisco.espinoza@unison.mx}
\address{Departamento de Matem\'aticas, Universidad de Sonora, M\'exico}

\author[M. E. Fr\'ias-Armenta]{Mart\'in Eduardo Fr\'ias-Armenta}
\email{eduardo.frias@unison.mx}
\address{Departamento de Matem\'aticas, Universidad de Sonora, M\'exico}

\author[H. A. Hern\'andez-Hern\'andez]{H\'ector Alfredo Hern\'andez-Hern\'andez}
\email{hector.hernandez@unison.mx}
\address{Departamento de Matem\'aticas, Universidad de Sonora, M\'exico}

\begin{abstract}
The notion of a contractible transformation on a graph was introduced by Ivashchenko as a means to study molecular spaces arising from digital topology and computer image analysis, and more recently has been applied to topological data analysis. Contractible transformations involve a list of four elementary moves that can be performed on the vertices and edges of a graph, and it has been shown by Chen, Yau, and Yeh that these moves preserve the simple homotopy type of the underlying clique complex.  A graph is said to be $\II$-contractible if one can reduce it to a single isolated vertex via a sequence of contractible transformations.  Inspired by the notions of collapsible and non-evasive simplicial complexes, in this paper we study certain subclasses of $\II$-contractible graphs where one can collapse to a vertex using only a subset of these moves.  Our main results involve constructions of minimal examples of graphs for which the resulting classes differ.  We also relate these classes of graphs to the notion of $k$-dismantlable graphs and $k$-collapsible complexes, which also leads to a minimal counterexample to an erroneous claim of Ivashchenko from the literature.  We end with some open questions.
\end{abstract}

\maketitle



\section{Introduction}\label{sec:intro}

In many applications of graph theory it is of interest to find combinatorial operations on a graph that preserve the topology of its clique complex. In \cite{Ivashchenko1994} Ivashchenko introduced what we will call {\it $\II$-contractible transformations}, a collection of four modifications one can make on the vertices and edges of a graph (see Definition \ref{IF} for a precise statement).  These local operations are used in computer image analysis, the theory of molecular spaces, and digital topology.  For instance to model a digital image $S$ embedded in $n$-dimensional Euclidean space one can divide ${\mathbb R}^n$ into a set of cubes of a certain scale and consider the \emph{molecular space} $M_1$ obtained by the set of cubes intersecting $S$.  Changing the scale of these cubes gives rise to another molecular space $M_2$.  Under certain conditions the intersection graphs of these two molecular spaces can be transformed into each other via a sequence of $\II$-contractible transformations.

Allowing for sequences of $\II$-contractible transformations defines an equivalence relation on the set of finite graphs. In particular the set of {\it $\II$-contractible} graphs, denoted $\II$, consists of those graphs that can be reduced to a single isolated vertex through such a sequence.  These notions mimic constructions in combinatorial topology and simple homotopy theory, where similar operations on simplicial (or more generally CW-) complexes are used to define various families within a homotopy class. In the case of graphs and their clique complexes, these moves can be described in terms of simple graph theoretic constructions that are more well-suited for computer implementation.  As such they have applications to topological data analysis and for instance the computation of persistent homology.

In \cite{Ivashchenko1994} Ivashchenko proved that $\II$-contractible transformations on a graph $G$ do not change the homology groups of its clique complex $\Delta(G)$.  More recently in \cite{Chen2001} Chen, Yau, and Yeh proved that $\II$-contractible transformations in fact preserve the \emph{simple homotopy type} of the underlying clique complex.  In particular if $G$ is an $\II$-contractible graph then $\Delta(G)$ is contractible.

 There are some important subfamilies of the class $\II$ obtained by allowing only certain subsets of the $\II$-transformations described above.  Again these mimic constructions for simplicial complexes where the classes of \emph{collapsible} and \emph{nonevasive} complexes are obtained by allowing only certain kinds of collapsing operations.  For graphs, we define the class $\SIC$ of  \defi{strong $\II$-contractible} graphs to be those obtained by only applying the \emph{gluing} type operations (I2 and I4 of Definition \ref{IF}).  In \cite{Espinoza2018} it is proved that if a graph $G$ is strong $\II$-contractible then the clique complex of $G$ is collapsible.   We define the class $\VIC$ of \defi{vertex $\II$-contractible} graphs by restricting to the transformations that only involve \emph{vertices} (I1 and I2 of Definition \ref{IF}).  Finally, the class $\SVIC$ of \defi{strong vertex $\II$-contractible} are those obtained by when we only allow vertex gluing (I2).

A natural question to ask is whether these actually constitute different classes of graphs; that is if there is any redundancy in the transformations described in Definition \ref{IF}.  In \cite{Chen2001} Chen, Yau, and Yeh show that in fact the class of $\II$-contractible and vertex $\II$-contractible graphs coincide. 
In particular they prove that both an edge deletion and an edge gluing can be realized by the composition of a vertex gluing followed by a vertex deletion. 
The authors also describe a flag triangulation of Bing's house which provides an example of a graph that is $\II$-contractible but not strong $\II$-contractible. In summary we have the following containments of graph classes.
\begin{equation}\label{eq:contain}
{\SVIC} \subsetneq {\SIC} \subsetneq {\II} = {\VIC}.
\end{equation}

In addition to \cite{Chen2001} a number of papers have studied these and related operations on graphs, but confusion has arisen surrounding exactly which transformations are needed to define the resulting classes.  Indeed, in a follow up paper \cite{Ivashchenko1994a} authored by Ivashchenko he claims that any $\II$-contractible graph can be obtained from an isolated vertex by a series of $\II$-contractible gluings of vertices.  It turns out that this is not the case, and a counterexample was described by the third author in \cite{FriasArmenta2020}. 

We will see that the class $\SVIC$ of strong vertex $\II$-contractible graphs also have close connections to the class of \emph{$k$-dismantlable} graphs as introduced by Fieux and Jouve in \cite{Fieux2020}.  This class of graphs generalize the well-studied notion of \emph{dismantlable} (here corresponding to $0$-dismantlable) graphs, which have seen applications in the study of homomorphism complexes \cite{Doc2009}, statistical physics and Gibbs measures \cite{Brightwell2000}, pursuit-evasion games on graphs \cite{Now1983}, iterated clique graphs \cite{Fri2004}, and chordal graphs \cite{Ada2017}.   The class of $k$-dismantlable graphs can also be seen as graph-theoretical analogues of the \emph{$k$-collapsible} simplicial complexes of Barmak and Minian \cite{Barmak2011}, who used these concepts to define a notion of \emph{strong homotopy type}. In \cite{Barmak2011} it is shown that a simplicial complex is non-evasive if and only if it is $k$-collapsible for some $k$.

\subsection{Our contributions}
In this paper we seek to further understand the distinction between the classes of graphs $\SVIC$, $\SIC$, and $\II$.  We construct minimal examples of graphs that demonstrate how these families differ, and how these classes relate to $k$-dismantlability and other constructions from the literature. Our examples were constructed by first creating a bank of all isomorphism types of connected graphs up to 11 vertices, and then employing the software Ripser \cite{Bauer2021} to compute homology of the corresponding clique complexes. Among those with vanishing homology, we tested for the various contractible properties by employing code available at \cite{gcs}. Our first result concerns contractible graphs that are not in $\SVIC$, addressing the first containment in \ref{eq:contain}.


\newtheorem*{homologia-trivial-nsvic}{Theorem \ref{homologia-trivial-nsvic}}
\begin{homologia-trivial-nsvic}
The smallest graphs (in terms of vertices) that have trivial homology but which are not strong vertex $\II$-contractible are depicted in Figure \ref{Grupo12}. Furthermore these graphs are all strong $\II$-contractible. Hence for graphs on at most ten vertices, the class of acyclic, $\II$-contractible, strong $\II$-contractible, and strong vertex $\II$-contractible graphs all coincide.
\end{homologia-trivial-nsvic}

We next consider graphs that are $\II$-contractible but not strong $\II$-contractible, addressing the second inequality in \ref{eq:contain}. As mentioned above, Chen, Yau, and Yeh \cite{Chen2001} describe a flag triangulation of Bing's house that provides such a graph on 21 vertices.  We construct an example of such a graph on 15 vertices via a flag triangulation of the \emph{Dunce Hat} (see Figure \ref{Dunce}).

\newtheorem*{prop:dunce}{Proposition \ref{prop:dunce}}
\begin{prop:dunce}
The graph depicted in Figure \ref{DunceA} is $\II$-contractible but is not strong-$\II$-contractible.
\end{prop:dunce}

We have not been able to verify whether this graph is minimal with respect to this property, although we do suspect that this is the case (see in Conjecture \ref{conj:dunce}).



By definition the class of strong vertex $\II$-contractible graphs involves removing vertices whose neighborhoods satisfy certain properties, and a natural question to ask is whether one can apply a greedy algorithm to perform these removals. In other words, if $G$ is strong vertex $\II$-contractible and $v \in G$ is an $\SVIC$-contractible vertex is it true that $G - v$ is strong vertex $\II$-contractible?  It turns out that this is not always the case, that in fact one can get ``stuck'' in the process of removing such vertices.  Our next result provides the smallest examples of graphs where this happens.

\newtheorem*{minima-de-12}{Theorem \ref{minima-de-12}}
\begin{minima-de-12}
The graphs depicted in Figure \ref{stuck} are the smallest graphs $G$ (in terms of vertices) that have the property that
\begin{enumerate}
    \item $G$ is strong vertex $\II$-contractible;
    \item There exists a vertex $v \in G$ such that $N_G(v)$ is strong vertex $\II$-contractible and yet $G-v$ is not strong vertex $\II$-contractible.
\end {enumerate}
\end{minima-de-12}

We next relate our constructions to the class of $k$-dismantlable graphs, as introduced by Fieux and Jouve in \cite{Fieux2020}. This notion again involves an inductive definition in terms of the removal of vertices satisfying certain properties, and one justification for the  study of $0$-dismantlable graphs is that order does \emph{not} matter in the removal of these vertices.  We provide minimal examples of graphs that are strong vertex $\II$-contractible but not $0$-dismantlable.

\newtheorem*{thm:notdis}{Theorem \ref{thm:notdis}}
\begin{thm:notdis}The smallest graphs (in terms of vertices) that are strong vertex $\II$-contractible but not 0-dismantlable are depicted in Figure \ref{fig:notdis}.  Furthermore these graphs are all $1$-dismantlable.
\end{thm:notdis}

We remark that two of these graphs were also described in \cite{Bou2010} and \cite{Fieux2020} as  examples of graphs that are 1-dismantlable but not 0-dismantlable, we thank the referee for pointing this out. The notion of $k$-dismantlability also has an analogue for simplicial complexes in the context of \emph{$k$-collapsibility} first introduced by Barmak and Minian in \cite{Barmak2011}, who prove that a simplicial complex is non-evasive if and only if it is $k$-collapsible for some $k$.  We prove an analogous fact for graphs.

\newtheorem*{sicne}{Theorem \ref{sicne}}
\begin{sicne}A graph $G$ is strong vertex $\II$-contractible if and only if $G$ is $k$-dismantlable for some $k$.
\end{sicne}

From this we see that a graph $G$ is in $\SVIC$ if an only if its clique complex $\Delta(G)$ is nonevasive. In a similar vein, in \cite{Fieux2020} Fieux and Jouve define a finite graph to be \emph{non-evasive} if it is $k$-dismantable for some $k$. Hence we conclude that non-evasive graphs coincide with the class $\SVIC$.

It turns out that the examples from Theorem \ref{thm:notdis} also relate to an erroneous ``axiom'' that Ivashchenko formulates in \cite{Ivashchenko1994a}.  Here it is claimed that if $G$ is an $\II$-contractible graph $G$ and if $v \in G$ is any non-cone vertex then one can find a vertex $u\in G$ such that $N_G(u,v)$ is $\II$-contractible (see Statement \ref{Iva_axiom} for a precise statement). In \cite{FriasArmenta2020} the third author describes a graph on 13 vertices that in fact contradicts the statement, and a smaller graph on 11 vertices was found by Ghosh and Ghosh in \cite{Ghosh2021}.  Our constructions provide minimal counterexamples.

\newtheorem*{thm:axiom}{Theorem \ref{thm:axiom}}
\begin{thm:axiom}
The two labeled graphs on 8 vertices depicted in Figure \ref{fig:notdis} are the smallest graphs (in terms of vertices) that contradict Axiom 3.4 from \cite{Ivashchenko1994a}. There are 133 graphs on $9$ vertices that contradict the Axiom.
\end{thm:axiom}

The rest of the paper is organized as follows.  In Section \ref{sec:definitions} we review some concepts from combinatorial topology and  define the various classes of graphs that we study.  In Section \ref{sec:notstrong} we provide examples of graphs that illustrate the hierarchy of contractible graphs depicted in Equation \ref{eq:contain}, and establish Theorem \ref{homologia-trivial-nsvic} and Proposition \ref{prop:dunce}. In Section \ref{sec:order} we discuss the relevance of vertex orderings for strong vertex $\II$-contractible graphs and prove Theorem \ref{minima-de-12}.   In Section \ref{sec:dismantlable} we discuss $k$-dismantlability of graphs, and establish Theorems \ref{thm:notdis} and \ref{sicne}. Here we also discuss minimal counterexamples to the Ivaschenko axiom and prove Theorem \ref{thm:axiom}. In Section \ref{sec:further} we discuss some open problems and future directions. 









\section{Definitions and preliminaries}\label{sec:definitions}


Here we collect some basic definitions and set some notation. Throughout the paper we use $G = (V,E)$ to denote a finite simple graph, with no loops and no multiple edges. We let $K(n)$ denote the \defi{complete graph} on $n$ vertices, so that in particular $K(1)$ denotes a single isolated vertex. If $G$ is a graph and $v \in V(G)$ is a vertex we let $N_G(v)$ (resp. $N_G[v]$) denote the \defi{open neighborhood} (resp. \defi{closed neighborhood}) of $v$ in $G$, defined by
\begin{align*}
N_G(v) &= \{ u \in V(G): \{u,v\} \in E(G)\}, \\
N_G[v] &= N_G(v) \cup \{v\}.
\end{align*}

If $v$ and $w$ are vertices of $G$, their \defi{common neighborhood} $N_G(v,w)$ is given by the intersection
\[N_G(v,w) = N_G(v) \cap N_G(w).\]
Given a graph $G$, we will often abuse notation and use $G$ to also denote its set of vertices. Similarly, given a subset of vertices $B\subset V(G)$, we will use $B$ to also denote the subgraph of $G$ induced on the elements of $B$. In particular, we use $N_G(v)$ to denote the set of neighbors of $v$, as well as the subgraph of $G$ induced by these vertices.  For a vertex $v \in G$ we let $G-v$ denote the subgraph of $G$ obtained by deleting $v$.   From \cite{Ivashchenko1994} we have the follow notion.

\begin{definition}\label{IF} The class of \defi{$\II$-contractible graphs}, denoted $\II$, is defined as follows.
\begin{enumerate}
    \item
    The graph $K(1)$, consisting of a single isolated vertex, is in $\II$.
\item A graph is in $\II$ if it can be obtained from an $\II$-contractible graph $G$ by a sequence of the following operations:

\begin{enumerate}
\item[(I1)] Deleting a vertex: A vertex $v$ of $G$ can be deleted if $N_G(v) \in \II$.
\item[(I2)] Gluing  a vertex: If $G^\prime$ is an $\II$-contractible subgraph of $G$, then a vertex $v \notin G$ can be glued to $G$ to produce a new graph $G^{\prime \prime}$ with $N_{G^{\prime \prime}}(v) = G^\prime$. 
\item[(I3)] Deleting  an edge: An edge $\{v_1,v_2\}$ of $G$ can be deleted if the common neighborhood satisfies $N_G(v_1,v_2) \in \II$.
\item[(I4)] \label{gluedge} Gluing an edge: For two non-adjacent vertices $v_1$ and $v_2$ of $G$, the edge $\{v_1,v_2\}$ can be glued to $G$ if the common neighborhood satisfies $N_G(v_1,v_2) \in \II$.
\end{enumerate} 
\end{enumerate}

We say that a vertex $v \in G$ (resp. an edge $\{v_1,v_2\}$) is \defi{$\II$-contractible} if $N_G(v)$ (resp. $N_G(v_1,v_2)$) is in $\II$.
\end{definition}

It is clear that all complete graphs are $\II$-contractible and more generally it is not hard to show that connected chordal graphs are also in the class $\II$. Two graphs $G$ and $H$ are said to be \defi{$\II$-homotopy equivalent} if one can be obtained from other via a sequence of the transformations (I1)--(I4).  One can see that $\II$-homotopy equivalence defines an equivalence relation on the set of finite graphs. A graph $G$ is said to be \defi{$\II$-contractible} if it is $\II$-homotopy equivalent to a single vertex.


Given a graph $G$ its \defi{clique complex} $\Delta(G)$ is by definition the simplicial complex on the vertex set $V(G)$ whose simplices are the complete subgraphs of $G$.  Note that $\Delta(G)$ is a \emph{flag} simplicial complex; that is, its minimal nonfaces have dimension 2. To save on notation we let $H_n(G;A)$ denote the $n$th homology group of the geometric realization clique complex of $G$ with coefficients in the abelian group $A$.  We let $H_n(G) := H_n(G;{\mathbb Z})$ denote the homology groups of $G$ with integer coefficients.  

\noindent
{\bf Convention.} We will often speak about topological properties (e.g. homology groups) of a graph $G$, by which we mean those of its clique complex $\Delta(G)$.

One can check that the transformations on a graph $G$ described in Definition \ref{IF} are \emph{functorial}, in the sense that they induce a continuous map on the realizations of the underlying clique complexes.  In \cite{Chen2001} Chen, Yau, and Yeh prove that the contractible transformations in Definition \ref{IF} preserve the simple homotopy type of the underlying clique complexeses.  In particular we get the following (a strengthening of the main result from \cite{Ivashchenko1994}).

\begin{corollary}\cite[Corollary 3.6]{Chen2001} \label{thm:contractible}
If $G$ is an $\II$-contractible graph then $\Delta(G)$ is contractible.
\end{corollary}

As far as we know the converse of Corollary \ref{thm:contractible} is still open; that is, if $\Delta(G)$ is contractible is not known if $G$ is necessarily $\II$-contractible.  See Section \ref{sec:further} for more discussion.


\subsection{Special subclasses of \texorpdfstring{$\II$}{I}-contractible graphs}

The definition of $\II$-contractible graphs in Definition \ref{IF} involves a list of four transformations, and a natural question to ask is where all are necessary to define the class. In \cite{Chen2001} Chen, Yau, and Yeh prove that in fact there is a redundancy. Although the statement of \cite[Lemma 3.4]{Chen2001} is misleading, the proof leads to the following fact. 

\begin{lemma}\cite[Lemma 3.4]{Chen2001}\label{chenlemma}
Edge deletion $(I3)$ and edge gluing $(I4)$ can each be realized by the composition of a vertex gluing $(I2)$ and vertex deletion $(I1)$.
\end{lemma}

It then follows that in our definition of $\II$-contractible graphs we in fact only needed the transformations (I1) and (I2) to define the class. In this paper we will primarily focus on classes of graphs defined by other subsets of the transformations described in Definition \ref{IF}. In particular we define the following classes.

\begin{definition}
The class of \defi{strong $\II$-contractible graphs}, denoted $\SIC$, is defined as follows:
\begin{enumerate}
\item The trivial graph $K(1)$ is in $\SIC$.
\item A graph is in $\SIC$ if it can be obtained from a strong $\II$-contractible graph $G$ by applying a sequence of the following transformations:

\begin{enumerate}
\item[(I2)] Gluing  a vertex:  If $G^\prime$ is a subgraph of $G$ satisfying $G^\prime \in \SIC$, then a vertex $v$ not in $G$ can be added to produce a graph $G^{\prime \prime}$ such that $N_{G''}(v)= G^\prime$;
\item[(I4)] Gluing an edge:  For two non-adjacent vertices $v_1$ and $v_2$ of $G$, the edge $\{v_1,v_2\}$ can be glued to $G$ whenever $N_G(v_1,v_2) \in \SIC$.
\end{enumerate} 
\end{enumerate}

We will say that a vertex $v$ (resp. edge $e$) in $G$ is  \defi{$\SIC$-contractible} if $N_G(v)$ (resp. $N_G(e)$) is in $\SIC$.
\end{definition}

\begin{definition}\label{SVIC} The class of \defi{strong vertex $\II$-contractible graphs}, denoted by $\SVIC$, is defined as follows:
\begin{enumerate}
\item The trivial graph $K(1)$ is in $\SVIC$.
\item A graph is in $\SVIC$ if it can be obtained from a strong vertex $\II$-contractible graph $G$ via a sequence of the following operation:

\begin{enumerate}
\item[(I2)] Gluing  a vertex:  If $G^\prime$ is a subgraph of $G$ satisfying $G^\prime \in \SVIC$ then a vertex $v$ not in $G$ can be added to produce a graph $G^{\prime \prime}$ such that $N_{G''}(v) = G^\prime$.

\end{enumerate} 
\end{enumerate}

We will say that a vertex $v \in G$ is \defi{$\SVIC$-contractible} if $N_G(v) \in \SVIC$.
 \end{definition}

A connected graph $G$ is said to be a \defi{cone} if there exists a vertex $v$ that is adjacent to all other vertices.  We observe that cones are strong vertex $\II$-contractible.  Indeed, for any vertex $w \neq v$ we have that both $N_G(w)$ and  $G - w$ are cones (with cone point $v$) and hence strong vertex $\II$-contractible by induction. 

\subsection{Simplicial complexes}
To motivate our graph-theoretic definitions we discuss some analogous constructions in the setting of simplicial complexes.  For this we recall the following notions.

\begin{definition}
Suppose $\Delta$ is a simplicial complex and $\sigma \in \Delta$ is a face.  Then the \defi{link}, \defi{deletion}, and \defi{face deletion} of $\sigma$ are the subcomplexes of $\Delta$ defined as
\begin{align*}
\mathrm{link}_{\Delta}(\sigma) &= \{\tau \in \Delta: \sigma \cup \tau \in \Delta, \sigma \cap \tau = \varnothing\}, \\
\mathrm{del}_{\Delta}(\sigma) &= \{\tau \in \Delta: \sigma \cap \tau = \varnothing\}, \\
\mathrm{fdel}_{\Delta}(\sigma) &= \{\tau \in \Delta: \sigma \nsubseteq \tau\}. 
\end{align*}
\end{definition}


The class of \emph{collapsible} simplicial complexes is typically defined in terms of free faces and the notion of elementary collapses.  We will use an equivalent formulation from \cite[Definition 3.14]{Jonsson2008}.

\begin{definition}\label{def:collapsible}
The  class of \defi{collapsible} simplicial complexes is defined recursively as follows:
\begin{itemize}
    \item The void complex $\varnothing$ and any 0-simplex $\{\varnothing, \{v\}\}$ are collapsible.
    \item If $\Delta$ contains a nonempty face $\sigma$ such that the face deletion $\text{fdel}_{\Delta}(\sigma)$ and $\mathrm{link}_\Delta(\sigma)$ are collapsible, then $\Delta$ is collapsible.
\end{itemize}
\end{definition}

One can see that collapsible complexes are contractible but the converse is not true. If we insist that the face $\sigma$ in the second condition of Definition \ref{def:collapsible} is a singleton, we obtain another well known class of simplicial complexes.

\begin{definition}
The class of \defi{nonevasive} simplicial complexes is defined recursively as follows:
\begin{itemize}
    \item The void complex $\varnothing$ and any 0-simplex $\{\varnothing, \{v\}\}$ are collapsible.
    \item If $\Delta$ contains a vertex $v$ such that the face deletion $\text{fdel}_{\Delta}(v)$ and $\mathrm{link}_\Delta(v)$ are collapsible, then $\Delta$ is collapsible.
\end{itemize}
\end{definition}

It is clear that nonevasive complexes are collapsible, but the converse is not true. With these definitions we see that the class of strong $\II$-contractible graphs can be thought of as a graph-theoretical analogue of collapsible simplicial complexes.   Indeed one can check that the clique complex of a $\SIC$-contractible graph is collapsible \cite{Espinoza2018}, although it is in an open question whether the reverse implication holds. We recall that if $\Delta$ is a flag simplicial complex then if $\sigma$ has $|\sigma| \geq 3$ then $\text{fdel}_{\Delta}(\sigma)$ is no longer flag, and hence the definitions do not directly translate.   In \cite{Espinoza2018} it is conjectured that if $\Delta(G)$ is a flag collapsible complex then $G$ is strong $\II$-contractible, and some computational evidence for this conjecture is provided.  We refer to Section \ref{sec:further} for more discussion.

Similarly, the class of strong vertex $\II$-contractible graphs can be seen as a graph theoretical analogue of nonevasive simplicial complexes.  Indeed we show in Section \ref{sec:dismantlable} that a flag simplicial complex $\Delta$ is nonevasive if and only if it is the clique complex of a strong vertex $\II$-contractible graph.  Hence  $\SVIC$ corresponds to the class of \emph{non-evasive} graphs, as defined in \cite{Fieux2020}.

\subsection{Dismantlings}

For our study we will need some further notions from combinatorial topology that have analogues in graph theory.  In \cite{Barmak2011} Barmak and Minian study the notion of a \emph{strong collapse} of a simplicial complex and related notions of \emph{$k$-collapsibility}. In what follows a vertex $v$ of a simplicial complex $\Delta$ is a \defi{simplicial cone} if there exists a vertex $v^\prime \in \mathrm{link}_{\Delta}(v)$ such that any simplex $\sigma \in \mathrm{link}_{\Delta}(v)$ can by written as $\{v'\} \cup \tau$, for some $\tau \in \mathrm{link}_{\Delta}(v)$. Following the notation in \cite{Barmak2011} we have the following definitions.

\begin{definition}
A vertex $v$ of a finite simplicial complex $\Delta$ is \defi{$0$-collapsible} if the $\mathrm{link}_{\Delta}(v)$ is a simplicial cone.  A simplicial complex $\Delta$ is \defi{$0$-collapsible} if we can apply a sequence of deleting $0$-collapsible vertices that ends in a single vertex.

Inductively we say that a vertex $v$ is \defi{$k$-collapsible} if the complex  $\mathrm{link}_{\Delta}(v)$ is $(k-1)$-collapsible. We say that $\Delta$ is \defi{$k$-collapsible} if we can obtain a single vertex via a sequence of deleting $k$-collapsible vertices.
\end{definition}

It is not hard to see that a $0$-collapsible complex is collapsible (and hence contractible). One motivation for studying $0$-collapsible complexes come from the fact that the order of deleting vertices in such a complex does not matter.  As is proved in \cite{Barmak2011}, if $\Delta$ is $0$-collapsible then a greedy algorithm of removing $0$-collapsible vertices will always end with a single vertex.  This is in contrast to the class of collapsible complexes, where one can ``get stuck'' in the process of performing elementary collapses.  In particular not all collapsible complexes are $0$-collapsible. The notion of $k$-collapsibility also relates to the notions discussed above, and for instance in \cite{Barmak2011} it is shown that if a  simplicial complex $\Delta$ is $k$-collapsible for some $k\geq 0$ then $\Delta$ is nonevasive.

These notions also have graph theoretical analogues first introduced in \cite{Fieux2020}.  Here a graph $H$ is a \defi{cone} if there exists a vertex $v$ with $N_H(v) = V(H) \backslash v$. 

\begin{definition}
For a graph $G$ and vertex $v \in G$, we say that $v$ is \defi{$0$-dismantlable} (or simply \defi{dismantlable} if the context is clear) if its open neighborhood $N_G(v)$ is a cone.  A graph $G$ is \defi{$0$-dismantlable} (or simply \defi{dismantlable}) if it can be reduced to a single vertex by successive deletions of $0$-dismantlable vertices.

Proceeding inductively, a vertex $v \in G$ is \defi{$k$-dismantlable} if its open neighborhood $N_G(v)$  is $(k-1)$-dismantlable, and a graph $G$ is \defi{$k$-dismantlable} if it can be reduced to a single vertex by successive deletions of $k$-dismantlable vertices.
\end{definition}



\section{Separating the hierarchy}\label{sec:notstrong}

In this section we discuss minimal examples of graphs that distinguish the classes in the chain of containments depicted in \ref{eq:contain}.  We first consider examples of graphs which are strong $\II$-contractible but not strong vertex $\II$-contractible. 


\begin{theorem}\label{homologia-trivial-nsvic}
The smallest graphs (in terms of vertices) that have trivial homology but which are not strong vertex $\II$-contractible are depicted in Figure \ref{Grupo12}. Furthermore these graphs are all strong $\II$-contractible. Hence for graphs on at most ten vertices, the class of acyclic, $\II$-contractible, strong $\II$-contractible, and strong vertex $\II$-contractible graphs all coincide.
\end{theorem}

\begin{proof}
To determine the graphs in Figure \ref{Grupo12}, we first obtained a list of all isomorphism types of graphs on at most 11 vertices from McKay's online collection \cite{McKay}. We then used an adaptation of the Ripser software \cite{Bauer2021} to determine those graphs whose clique complexes have trivial homology.  We then applied Algorithm \ref{SIVC} to determine those graphs that are not in $\SVIC$, the result is depicted in Figure \ref{Grupo12}.  

To verify that the graphs in Figure \ref{Grupo12} are indeed strong $\II$-contractible, we first note that an arrow indicates that one can pass from a graph to another by adding an $\II$-contractible edge.  One can check that the source elements in the resulting figure (the graphs along the bottom row) have the property removing any of the $\II$-contractible edges (depicted in red), results in a graph that has an $\II$-contractible vertex. 

In particular for the graph on the bottom left the removal of the $\II$-contractible edge $46$ leads to the $\II$-contractible vertex $5$. For the next two graphs in the bottom row, the removal of the edge $37$ leads to the $\II$-contractible vertex $7$.  For the next graph, removing the edge $47$ leads to the $\II$-contractible vertex $6$. Finally for the graph on the bottom right we delete the edge $49$ to obtain the $\II$-contractible vertex $11$.  

Among graphs that have 11 or fewer vertices, we used software to check that all those with acyclic clique complexes are in fact in the class $\SIC$. This establishes the last claim.
\end{proof}

\begin{algorithm}
    \SetKwInOut{Input}{Input}
    \SetKwInOut{Output}{Output}

    \Input{A graph $G$ and the cardinality $n$ of the vertex set.}
    \Output{The logical \texttt{TRUE} if $G \in \SVIC$, or \texttt{FALSE} otherwise.}
    \eIf{$n = 0$}
      {
        return \texttt{FALSE}
      } 
      {
            \eIf{$n = 1$}
      {
        return \texttt{TRUE}
      } {
        \For{$i \leftarrow 1$ \KwTo $n$}{
                \If{\normalfont \texttt{SIVcontractible.graph}($N_G(v_i), |N_G(v_i)|$) = \texttt{TRUE}}{
            \Return \texttt{SIVcontractible.graph}($G- v_i, n-1$)
            }
        }
    
        }
      \Return \texttt{FALSE}
      }
          
    \caption{\texttt{SIVcontractible.graph}}
    \label{SIVC}
\end{algorithm}


  We expect that there exist graphs that have trivial homology but which are not $\II$-contractible (for instance the 1-skeleton of a flag triangulation of the 2-skeleton of the Poincar\'e homology sphere) but these will all have many more vertices than the graphs depicted in Figure \ref{Grupo12}. 


\begin{figure}[htbp] 
\centering
  \begin{subfigure}{0.33\textwidth}
    \centering
    \includegraphics[width=\linewidth]{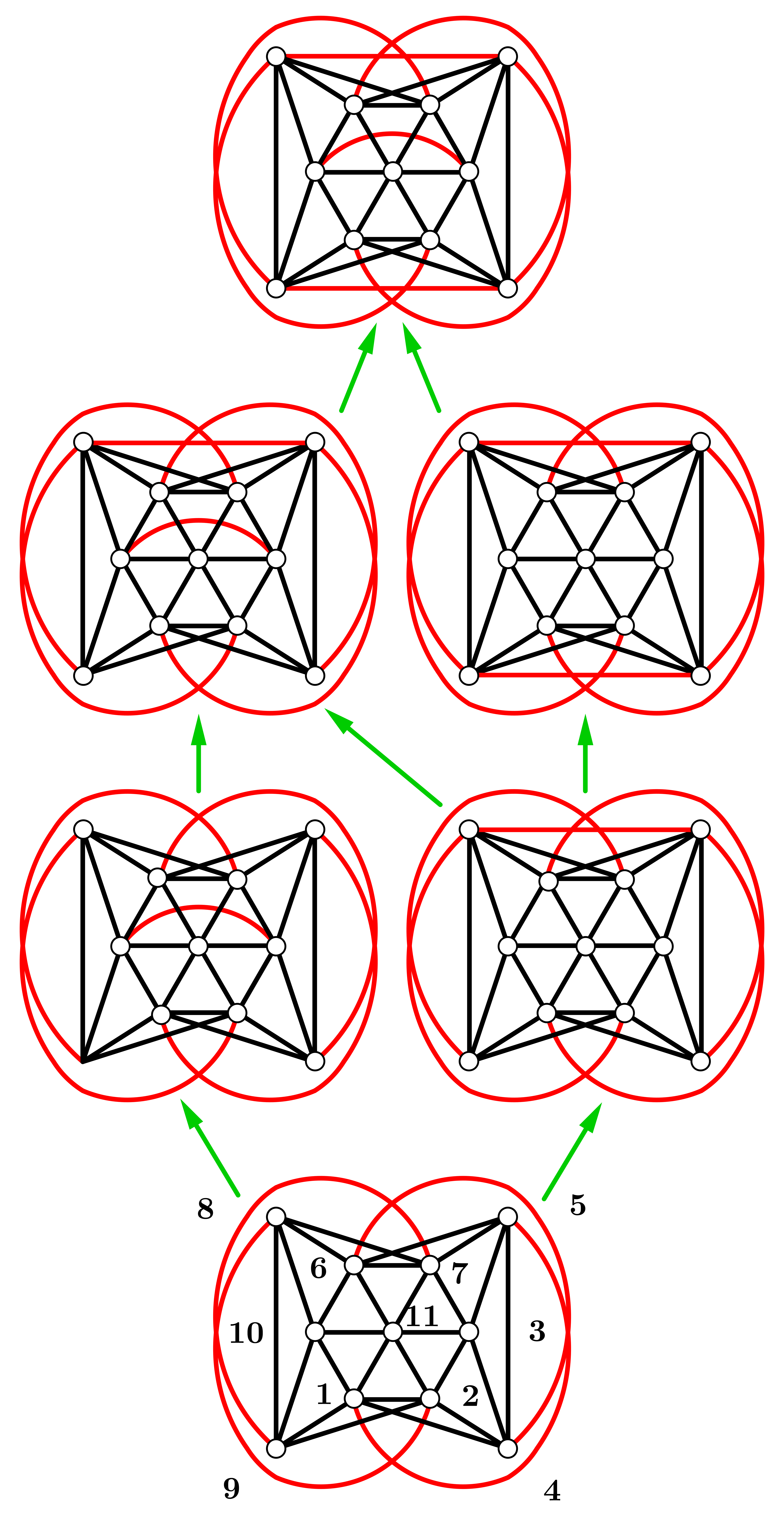}
    \end{subfigure}\hfill
  \begin{subfigure}{0.33\textwidth}
    \centering
    \includegraphics[width=0.95\linewidth]{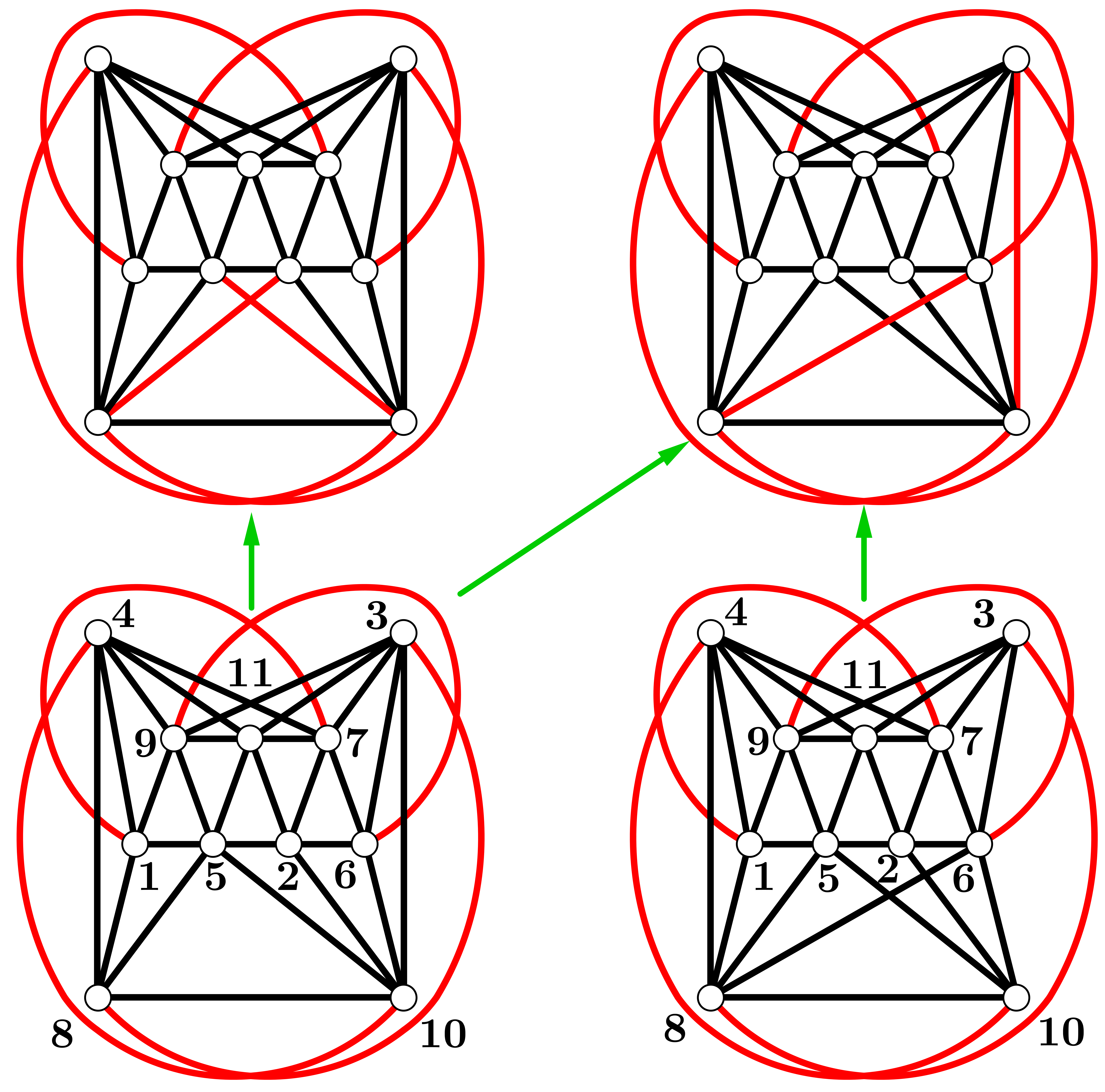}
  \end{subfigure}
  \begin{subfigure}{0.33\textwidth}
    \centering
    \vspace{70pt}
    \includegraphics[width=\linewidth]{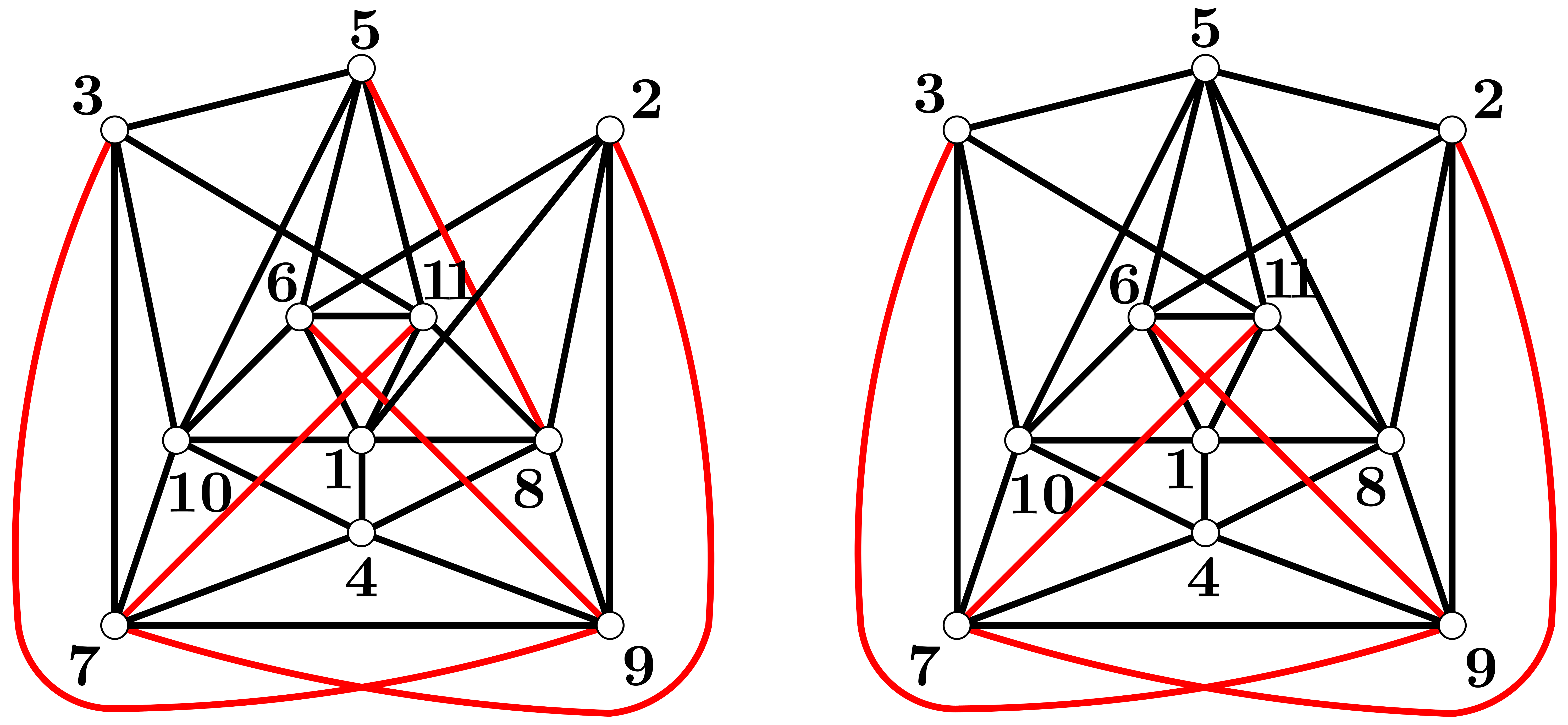}
  \end{subfigure}
  \caption{Minimal examples of graphs that are in $\SIC$ but not in $\SVIC$. An arrow indicates that we can pass from one graph to the other by adding an $\II$-contractible edge according to rule I4 of Definition \ref{IF}.}
  \label{Grupo12}
\end{figure}





\begin{remark}
From Theorem \ref{homologia-trivial-nsvic} it follows that the family of graphs on 11 vertices with trivial homology differs from the family of strong vertex $\II$-contractible graphs on 11 vertices.  In particular, any graph in Figure \ref{Grupo12} has trivial homology but is not in $\SVIC$. 

On the other hand, we observe that any strong vertex $\II$-contractible graph $G$ on at most 11 vertices has the following greedy property: the removal of \emph{any} $\SVIC$-contractible vertex will lead to a strong vertex $\II$-contractible graph, where again the removal of any $\SVIC$-contractible vertex will result in a graph in $\SVIC$.  In other words the choice of $\SVIC$-contractible vertices does not matter as we collapse to a single vertex. This follows from the fact that for graphs with at most 10 vertices, the classes of $\SIC$ and $\SVIC$ coincide. 


In addition, there exist graphs on 11 vertices that are in $\SIC$ but not $\SVIC$, but which become strong vertex $\II$-contractible with this greedy property after the removal of any $\SIC$-contractible edge (see for instance the five graphs in bottom in Figure \ref{Grupo12}, the label one).  We will see in Section \ref{sec:order} that the order of vertex removals does matter for strong vertex $\II$-contractible graphs with at least 12 vertices.
\end{remark}





We next consider graphs that are $\II$-contractible but not strong $\II$-contractible, addressing the second containment depicted in \ref{eq:contain}.  In \cite{Chen2001} a construction of such a graph on 21 vertices is described, based on ``the house with 2 rooms'' (or \emph{Bing's house}).  Recall that Bing's house is a 2-dimensional simplicial complex that is contractible but not collapsible, and hence (according to our discussion above) is a natural place to look for such an example.   We will be interested in graphs that are minimal with this property. By considering a triangulation of the Dunce Hat we obtain a graph on $15$ vertices with this property.

\begin{proposition}\label{prop:dunce}
The graph in Figure \ref{DunceA} is $\II$-contractible  but is not strong $\II$-contractible.
\end{proposition}

 Unfortunately our computational power is exhausted on graphs up to eleven vertices.  In particular we do not know if there are graphs on $n = 12,13,14$ vertices that are $\II$-contractible but not strong $\II$-contractible, although we conjecture that this is not the case.
 
\begin{figure}[htbp] 
\centering
  \begin{subfigure}{.5\textwidth}
    \centering

    \includegraphics[width=.825\linewidth]{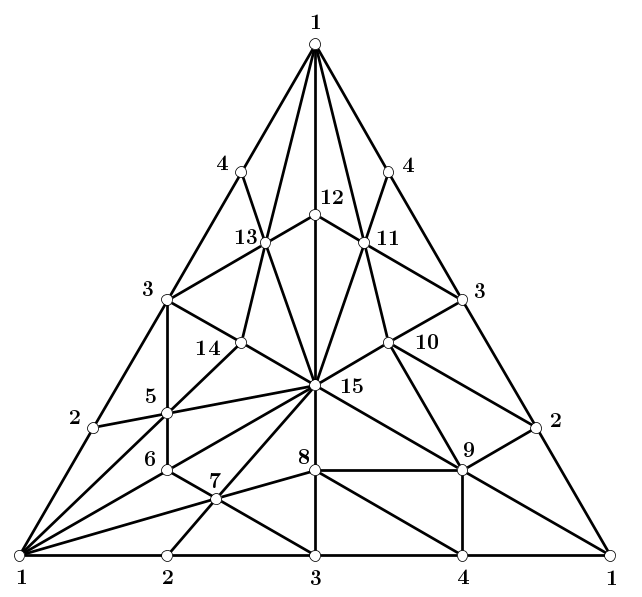}
    \caption{}
    \label{DunceA}
    \end{subfigure}\hfill
  \begin{subfigure}{.5\textwidth}
    \centering
    \includegraphics[width=.825\linewidth]{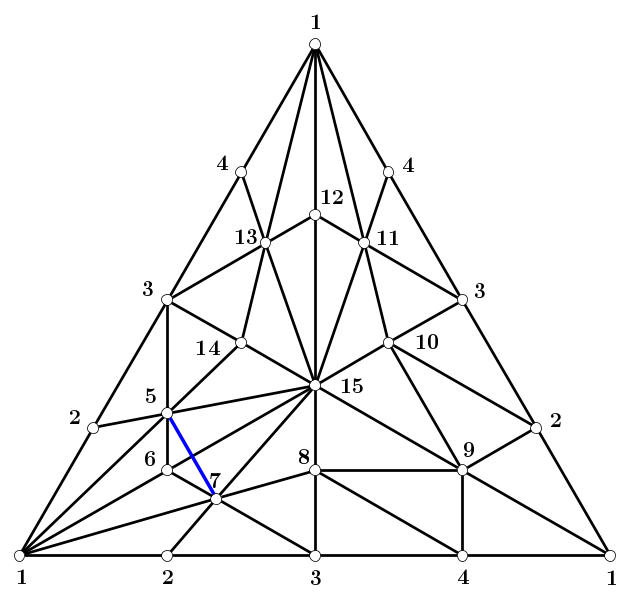}
    \caption{}
    \label{DunceB}
  \end{subfigure}
  \caption{{Graphs related to the notion of strong $\II$-contractibility (with identifications made as indicated).  The graph on the left is in $\II$ but not $\SIC$. The graph on the right is in $\SIC$ and is obtained from the first graph by adding the $\II$-contractible edge $\{5,7\}$.}
  \label{Dunce}}
\end{figure}

\begin{conjecture}\label{conj:dunce}
The smallest graph (in terms of vertices) with trivial homology that it is not strong $\II$-contractible is shown in Figure \ref{DunceA}. In particular this is the smallest graph that is $\II$-contractible but not strong $\II$-contractible. 
\end{conjecture}

We note that if we add the edge $\{5,7\}$ to the graph in Figure \ref{DunceA} according to rule $(I4)$ we obtain the graph in Figure \ref{DunceB}, which can be seen to be strong $\II$-contractible.  Indeed one can see that vertex $6$ is $\SIC$-contractible and a deleting order is given by: vertex $6$, edge $\{1,5\}$, edge $\{1,7\}$, and vertices $2,9,8,3,4,1$, at which point we obtain a graph that is a cone on $15$.

Hence in this case we only need to ``go up'' one step in order to obtain a graph that can be collapsed via removing $\SIC$-contractible edges and vertices. We furthermore conjecture that the graph in Figure \ref{DunceB} is a minimal example of this phenomenon in the following sense.

\begin{conjecture}\label{conj:dunce2}
The graph depicted in Figure \ref{DunceB} is the smallest graph $G$ (in terms of vertices) that has the following properties:
\begin{enumerate}
    \item $G$ is strong $\II$-contractible;
    \item There exists an edge $\{v,w\}$ in $G$ such that $N({v,w})\in \SIC$ and such that $G \backslash \{v,w\}$ is not strong $\II$-contractible.
\end{enumerate}
\end{conjecture}


\begin{algorithm}
    \SetKwInOut{Input}{Input}
    \SetKwInOut{Output}{Output}

    \Input{A graph $G$ and the cardinality $n$ of the vertices set.}
    \Output{The logical \texttt{TRUE} if $G \in \SIC$, or \texttt{FALSE} otherwise.}
    \eIf{$n = 0$}{
      
        \Return \texttt{FALSE}
      }{
      
       \eIf{$n = 1$}{
        \Return \texttt{TRUE}
       }{ \# Delete $\SIC$-contractible vertices\\
           \For{$i\leftarrow 1$ \KwTo $n$}{
             \If{\normalfont \texttt{SIcontractible.graph($N_G(v_i)$, $|N_G(v_i)|$)=TRUE}}{
              \Return \texttt{SIcontractible.graph($G-v_i$, $n-1$)}
             }
          }
          \# Delete $\SIC$-contractible edges\\
           \For{$i\leftarrow 1$ \KwTo $n-1$}{
             \For{$j\leftarrow i+1$ \KwTo $n$}{
               \If{\normalfont $G_{ij}=1$ \& \texttt{SIcontractible.graph}$(N_G(v_i,v_j))$=\texttt{TRUE}}{
                 \Return \texttt{SIcontractible.graph}$(G-(v_i,v_j),n)$\\
               }
          }
          }
          \Return \texttt{FALSE}
       }
      }
    \caption{\texttt{SIcontractible.graph}}
    \label{SIC}
\end{algorithm}
\newpage

\begin{algorithm}
    \SetKwInOut{Input}{Input}
    \SetKwInOut{Output}{Output}
    \Input{A graph $G \in \II$ and the cardinality $n$ of the vertices set.}
    \Output{The logical \texttt{TRUE} if $G$ satisfies the axiom, or \texttt{FALSE} otherwise.}
    $nVNC \leftarrow 0$\;
    \For{$i \leftarrow 1$ \KwTo $n$}{
        $grade\leftarrow 0$\;
        \For{$j \leftarrow 1$ \KwTo $n$}{
            \If{\normalfont $G_{ij}=1$}{
                $grade \leftarrow grade+1$\;
            }
        }
        \If{($grade<n$) }{
        $nVNC\leftarrow nVNC+1$\;
        $VNC_{nVNC}\leftarrow i$\;
        }
    }
    \texttt{Check axiom:}\;
    \For{$i \leftarrow 1$ \KwTo $nVNC-1$}{
        $CanAddEdge \leftarrow 0$\;
        \For{$j \leftarrow i+1$ \KwTo $nVNC$}{
            \If{\normalfont $G_{VNC_i,VNC_j}=0$}{
                \If{\normalfont \texttt{SIcontractible.graph}($N(G,n,VNC_i,VNC_j)$)=\texttt{TRUE}}{
                    $CanAddEdge\leftarrow 1$\;
                }
            }
        }
        \If{(\normalfont $CanAddEdge=0$)}{
            \Return \texttt{FALSE}
        }
    }
    \Return \texttt{TRUE}
    \caption{\texttt{CheckAxiom}}
    \label{CheckAxiom}
\end{algorithm}


\section{Order matters for strong vertex \texorpdfstring{$\II$}{I}-contractible graphs}\label{sec:order}

Recall that a graph $G$ is strong vertex {$\II$}-contractible if it can be obtained from a single isolated vertex by a sequence of gluings along $\SVIC$-contractible vertices (See Definition \ref{SVIC}).  Hence $G$ can be reduced to a single vertex by a sequence of vertex deletions, and a natural question to ask is whether the ordering of vertices matters in the deletion process.  More specifically, given a strong vertex {$\II$}-contractible graph $G$ and an $\SVIC$-contractible vertex $v$ (where $N_G(v)$ is strong vertex $\II$-contractible), is it always true that $G - v$ is strong vertex {$\II$}-contractible?  


In \cite{Fieux2020} Fieux and Jouve describe a graph on $16$ vertices that is strong vertex $\II$-contractible, but where the order of vertex deletion matters. In this section we describe the smallest graphs that have this property.  





\begin{figure}[htbp]
    \centering
    \begin{subfigure}{.325\textwidth}
        \centering
        \includegraphics[width=\linewidth]{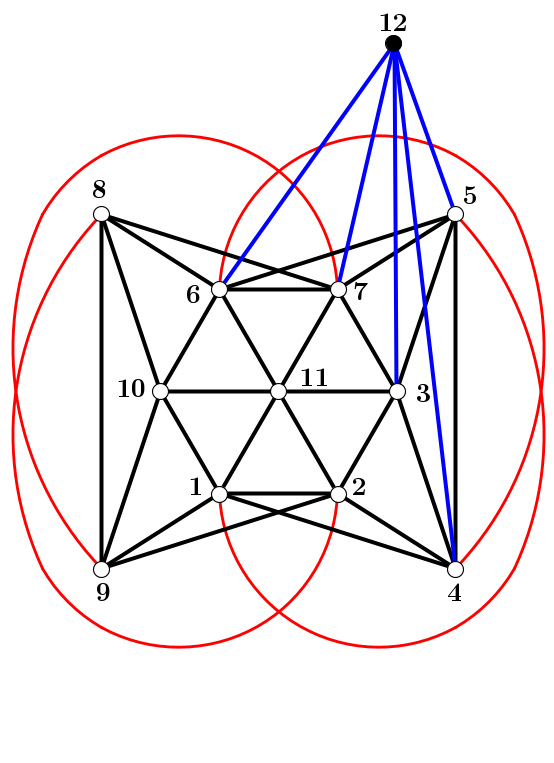}
        \caption{}
        \label{stuckA}
    \end{subfigure}\hspace{40pt}
    \begin{subfigure}{.325\textwidth}
        \centering
        \includegraphics[width=\linewidth]{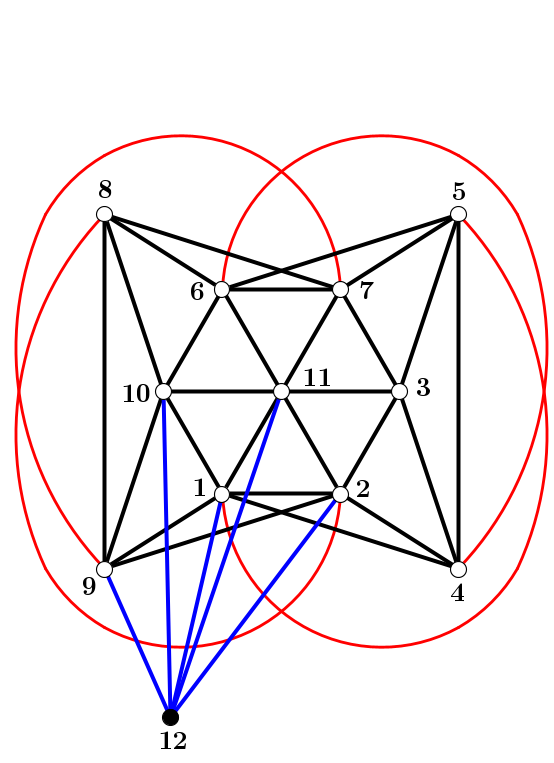}
        \caption{}
        \label{stuckB}
    \end{subfigure}
    \caption{Minimal examples of graphs that are strong vertex $\II$-contractible but where the order of $\SVIC$-contractible vertices matters.}
    \label{stuck}
\end{figure}

\begin{theorem}\label{minima-de-12}
The graphs depicted in Figure \ref{stuck} are the smallest graphs $G$ (in terms of vertices) that have the property that
\begin{enumerate}
    \item $G$ is strong vertex $\II$-contractible (that is, $G$ is in $\SVIC$);
    \item There exists a vertex $v$ such that $N_G(v)$ is strong vertex $\II$-contractible and yet $G-v$ is not in $\SVIC$.
    
    \end{enumerate}
Moreover any graph on 12 vertices and 35 edges with these properties is isomorphic to one of these.
\end{theorem}

For the proof of Theorem \ref{minima-de-12} we will need an auxiliary lemma regarding the automorphism group of a related graph.

\begin{lemma}\label{simetrias}
The graph $G_{30}$ depicted in Figure \ref{G30} has automorphism group given by ${\mathbb Z}_2 \times {\mathbb Z}_2$.
\end{lemma}

\begin{proof}
We let $G_{30}$ denote the graph depicted in Figure \ref{G30}. Note that there are four isomorphism classes of open vertex neighborhoods, given by $\{1, 2,6,  7\}$, $\{4,5,8,9\}$, $\{3,10\}$ and $\{11\}$.  We partition the vertex set of $G_{30}$ accordingly. 
Any automorphism of $G_{30}$ must fix those sets and must fix rigidly the hexagon $\{1, 2, 3, 6, 7, 10, 11\}$ depicted in Figure \ref{G30}. Therefore the only possible symmetries are the vertical, horizontal and point reflections (rotation by 180 degrees). We then have that $\bigl(\begin{smallmatrix}
    1 & 2 & 3 & 4 & 5 & 6 & 7 & 8 & 9 & 10 & 11 \\
    1 & 2 & 3 & 4 & 5 & 6 & 7 & 8 & 9 & 10 & 11
  \end{smallmatrix}\bigr)$, $\bigl(\begin{smallmatrix}
    1 & 2 & 3 & 4 & 5 & 6 & 7 & 8 & 9 & 10 & 11 \\
    2 & 1 & 10 & 9 & 8 & 7 & 6 & 5 & 4 & 3 & 11
  \end{smallmatrix}\bigr)$, $\bigl(\begin{smallmatrix}
    1 & 2 & 3 & 4 & 5 & 6 & 7 & 8 & 9 & 10 & 11 \\
    6 & 7 & 3 & 5 & 4 & 1 & 2 & 9 & 8 & 10 & 11
  \end{smallmatrix}\bigr)$, $\bigl(\begin{smallmatrix}
    1 & 2 & 3 & 4 & 5 & 6 & 7 & 8 & 9 & 10 & 11 \\
    7 & 6 & 10 & 8 & 9 & 2 & 1 & 4 & 5 & 3 & 11
  \end{smallmatrix}\bigr)$ are the automorphisms of the graph $G_{30}$. One can check that this set forms a group isomorphic to ${\mathbb Z}_2 \times {\mathbb Z}_2$.
\end{proof}

\begin{figure}[htbp] 
    \centering
    \includegraphics[width=0.325\linewidth]{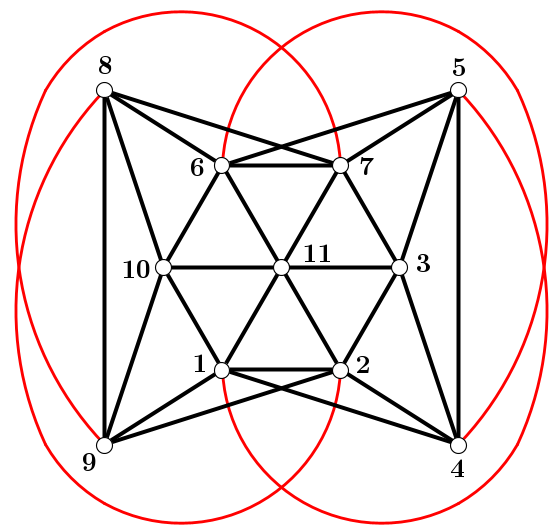}
    \caption{The graph $G_{30}$, the smallest graph (in terms of edges) among those depicted in Figure \ref{Grupo12}.}
    \label{G30}
\end{figure}

\begin{remark}\label{ConstyMin}
Before turning to the proof of Theorem \ref{minima-de-12}, we make some observations regarding the graphs in Figure \ref{stuck} and Figure \ref{G30}.

\begin{enumerate}
    \item Note that the edge $\{1,5\}$ in $G_{30}$ (see Figure \ref{G30}) is $\II$-contractible since its neighborhood is the single vertex $\{4\}$.  Recall from Lemma \ref{chenlemma} that one can realize its deletion as a vertex gluing followed by a vertex deletion.  The vertex gluing can be done in two ways (resulting in the graphs depicted in Figure \ref{stuck}), either by creating a neighbor of 5 (Figure \ref{stuckA}) or by creating a neighbor of 1 (Figure \ref{stuckB}).  Note that in the second case we also omit the edge $\{4, 12\}$.

    \item We can similarly construct new graphs from $G_{30}$ using instead any of the edges $\{1,5\}$, $\{2,8\}$, $\{6,4\}$ or $\{7,9\}$.  However, by the symmetry described in Lemma \ref{simetrias}, we will always obtain a graph isomorphic to one of the graphs depicted in Figure \ref{stuck}.
    
    \item\label{parte3} If we use this construction on any other graph from Figure \ref{Grupo12} we obtain a graph with more edges than those graphs depicted in Figure \ref{stuck}. 
\end{enumerate}
\end{remark}

\begin{proof}[Proof of Theorem \ref{minima-de-12}]

We first claim that each graph depicted in  Figure \ref{stuck} is strong vertex $\II$-contractible.  For this note that deleting the $\SVIC$-contractible vertex $5$ (for graph \ref{stuckA}) and vertex $1$ (for graph \ref{stuckB}) results in a graph isomorphic to $G \backslash \{1,5\}$.  By Remark \ref{ConstyMin} we see that the resulting graph is $\II$-contractible, and Theorem \ref{homologia-trivial-nsvic} implies that it is in fact in $\SVIC$. 

For the next property, we let $v = 12$ for both graphs depicted in Figure \ref{stuck}.  Indeed note that in each case $v$ is $\SVIC$-contractible, since its neighborhood is a cone (the edges incident to vertex $12$ are indicated in blue).  If one deletes the vertex $12$ (in both cases) we obtain the graph $G_{30}$ depicted in Figure \ref{G30} (also appearing in the bottom left in Figure \ref{Grupo12}). We have seen in Theorem \ref{homologia-trivial-nsvic} that this graph is not strong vertex $\II$-contractible, and hence one cannot continue with vertex deletions. This establishes the second of the two properties.

For the minimality claims, first note that for a graph $G$ satisfying the desired properties we need a vertex $v$ such that $G - v$ is strong $\II$-contractible but not in $\SVIC$.  Hence by Theorem \ref{homologia-trivial-nsvic} such a graph $G$ must have at least $12$ vertices. Among graphs with $12$ vertices that satisfy the properties, Remark \ref{ConstyMin}.(\ref{parte3}) implies that the graphs depicted in Figure \ref{stuck} are minimal with respect to edges.

\end{proof}


\section{\texorpdfstring{$k$}{k}-dismantlable graphs and  \texorpdfstring{$k$}{k}-collapsible complexes} \label{sec:dismantlable}

In this section we discuss how our constructions relate to the class of $k$-dismantlable graphs, as defined in Section \ref{sec:definitions}.   For the case of 0-dismantlable graphs it is known that the order of removing dismantlable vertices does not matter  \cite[Obs. 4.14]{FriasThesis}, and in fact the set of all dismantlings of a finite graph forms a \emph{greedoid} \cite{KorLov}. As discussed in Section \ref{sec:intro}, a similar fact holds for the class of strong collapsible simplicial complexes, see Theorem 2.11 from \cite{Barmak2011}.  

In Section \ref{sec:order} we saw that order \emph{does} in general matter for removing $\SVIC$-contractible vertices of strong vertex $\II$-contractible graphs.  As a special case of \cite[Proposition 5.1]{Bou2010} we observe the following.

\begin{lemma}\label{lem:dismantstrong}
If $G$ is a $0$-dismantlable graph then $G$ is strong vertex $\II$-contractible.
\end{lemma}

\begin{proof}
We use induction on $n = |V(G)|$.  The statement is clearly true for $n = 1$.  For $n > 1$ suppose $v \in V(G)$ is $0$-dismantlable. Then by definition $N_G(v)$ is a cone, which is strong vertex $\II$-contractible. The deletion $G - v$ is by definition $0$-dismantlable, which by induction is also strong vertex $\II$-contractible.  We conclude that $G$ is in $\SVIC$.
\end{proof}

Hence it is of interest to find examples of graphs that are in $\SVIC$ but not $0$-dismantlable.  Our next result describes the smallest such graphs.

\begin{theorem}\label{thm:notdis}
The smallest graphs (in terms of vertices) that are  strong vertex $\II$-contractible but not $0$-dismantlable  have $8$ vertices, and are depicted in Figure \ref{fig:notdis}. Furthermore these graphs are all $1$-dismantlable.
\end{theorem}

\begin{proof}
Again we verify this by explicitly checking all graphs with $8$ or fewer vertices.  We first consider all graphs that are acyclic and among those check which are $0$-dismantable using script available at \cite{gcs}.  Recall from Theorem \ref{homologia-trivial-nsvic} that any acyclic graph on $8$ or fewer vertices is in $\SVIC$. An explicit computation also shows that each graph is $1$-dismantlable. In particular every graph in Figure \ref{fig:notdis} has a vertex (depicted in bold) whose neighborhood is the path $P_4$ on $4$ vertices (which itself is $0$-dismantlable). One can check that the removal of any such vertex results in a $0$-dismantlable graph.
\end{proof}

We note that two of these graphs have appeared previously in the literature.  In particular the graph with $17$ edges depicted on the bottom left in Figure \ref{fig:notdis} was described in \cite{Bou2010}, and the graph with $18$ edges just above it was used in \cite{Fieux2020}, in both cases as examples of 
a graph that is $1$-dismantlable but not $0$-dismantlable.  We thank the referee for pointing this out to us.


\begin{figure}[htbp]
    \centering
    \includegraphics[width = 0.825\textwidth]{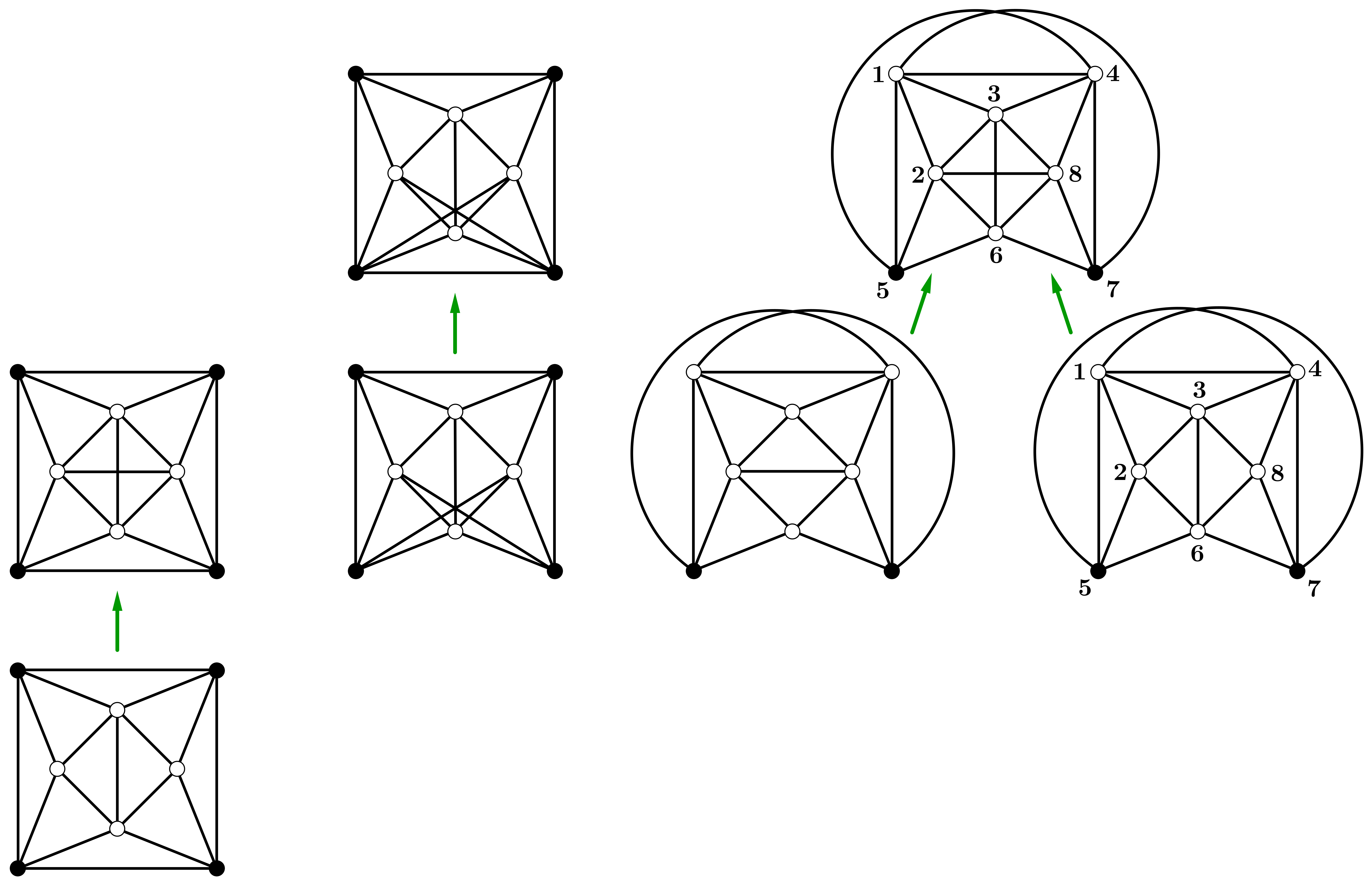}
    \caption{The smallest graphs that are $1$-dismantlable (and hence strong vertex $\II$-contractible) but not $0$-dismantlable. The vertices depicted in bold are $1$-dismantlable, and a green arrow indicates the addition of an $\II$-contractible edge. The two graphs with integer-labeled vertices are the smallest graphs that contradict Statement \ref{Iva_axiom}.}
    \label{fig:notdis}
\end{figure}

Recall from Section \ref{sec:definitions} that the notion of a $k$-collapsible vertex of a simplicial complex was defined by Barmak and Minian in \cite{Barmak2011}.  The definition is very similar to that of a $k$-dismantlable vertex of a graph, and in fact the two concepts coincide for flag simplicial complexes.

\begin{lemma}\cite[Proposition 4]{Fieux2020}
For a graph $G$, a vertex $v \in G$ is $k$-dismantlable if and only if $v$ is $k$-collapsible in the clique complex $\Delta(G)$.
\end{lemma}

From this we get the following characterization of $k$-dismantlable graphs.

\begin{proposition}\cite[Proposition 4]{Fieux2020} \label{kd-kc}
A graph $G$ is $k$-dismantlable if and only if $\Delta(G)$ is $k$-collapsible.
\end{proposition}

\begin{proof}
This follows from the fact that for any vertex $v \in G$ we have  $\mathrm{link}_{\Delta(G)}(v) = \Delta(N_G(v))$.
\end{proof}

Recall from \cite{Barmak2011} that a simplicial complex is non-evasive if and only if it is $k$-collapsible for some $k$. Also described in \cite{Barmak2011}, there exists simplicial complexes that are non-evasive but not $0$-collapsible.  Our Theorem \ref{thm:notdis} provides a minimal example among flag complexes.  Our next result relates these notions to the class $\SVIC$.


\begin{theorem}\label{sicne}
A graph $G$ is strong vertex $\II$-contractible if and only if $G$ is $k$-dismantlable for some $k$.
\end{theorem}

\begin{proof}
Suppose $G$ is a $k$-dismantlable graph. We will prove by induction on $k$ that $G\in\SVIC$. For $k=0$, we have from Lemma \ref{lem:dismantstrong} that $G$ is in $\SVIC$. Suppose the result is true for $k$ and suppose that $G$ is $(k + 1)$-dismantlable. By assumption we can reduce $G$ to a single vertex by deleting $k$-dismantlable vertices, that is, vertices $v$ whose open neighborhood $N(v)$ is a $k$-dismantlable graph.  Then by induction we have that each $N(v) \in \SVIC$, and thus $G \in \SVIC$.

Now suppose that $G \in \SVIC$.  We prove that $G$ is $k$-dismantlable, for some $k$, by induction $n = |V(G)|$. If $n = 1$ the claim is clear so we assume $n > 1$.  By assumption we have some vertex $v \in G$ such that $N_G(v) \in \SVIC$.  Hence by induction we have that $v$ is $\ell$-dismantlable for some $\ell$.  Also, since $G - v$ is in $\SVIC$ we have that $G-v$ is $\ell^\prime$-dismantlable for some $\ell^\prime$. Recall that if $G$ is $m$-dismantlable then $G$ is $m^\prime$-dismantlable for all $m^\prime \geq m$ (see \cite[Proposition 2]{Fieux2020}). Hence if we let $k = 1+\max\{\ell, \ell^\prime\}$, we conclude that $G$ is $k$-dismantlable.
\end{proof}

\begin{remark} 
With Theorem \ref{sicne} we can reformulate Theorems \ref{homologia-trivial-nsvic} and \ref{minima-de-12} in terms of $k$-dismantlability as follows.

\begin{itemize} 

\item In Figure \ref{Grupo12} we have the smallest graph $G$ (in terms of vertices) that has the properties: 

\begin{enumerate} 

    \item Its clique complex $\Delta(G)$ has trivial homology; 

    \item $G$ is not $k$-dismantlable for any $k$ (and hence $G$ is not non-evasive). 

    \end{enumerate}

\item In Figure \ref{stuck} we have the smallest graph $G$ (in terms of vertices) that has the properties: 

\begin{enumerate} 

    \item $G$ is $j$-dismantlable for some $j$; 

    \item There exists a vertex $v \in G$ such that $N_G(v)$ is $k$-dismantlable for some $k$, but where $G-v$ is not $\ell$-dismantlable for any $\ell$. 

    \end{enumerate}

\noindent
In \cite{Fieux2020} the authors also construct a graph $P$ on 16 vertices containing a vertex $v$ such that $N_P(v)$ is $0$-dismantlable, and such that $P$ is 1-dismantlable and yet $P-v$ (which they call the `parasol graph') is not $\ell$-dismantlable for any $\ell$.

\end{itemize} 

\end{remark}

Our examples from Theorem \ref{thm:notdis} also relate to an erroneous ``axiom'' from \cite{Ivashchenko1994a} regarding $\II$-contractible graphs.  We recall the statement here.

\begin{axiom} \cite[Axiom 3.4]{Ivashchenko1994a} \label{Iva_axiom}
Suppose  that  $G$ is  an  $\II$-contractible  graph,  and  let  $v \in V(G)$  be a vertex that is not  adjacent to  some vertex  of  $G$.  Then  there  exists a vertex $u$, not adjacent to $v$, such  that  the subgraph  $N_G(v,u)$  is  $\II$-contractible.
\end{axiom}

Ivashchenko uses Statement \ref{Iva_axiom} to establish other false properties of $\II$-contractible graphs, for example that any $\II$-contractible graph can be obtained from an isolated vertex by only allowing contractible gluings of vertices.
In \cite{FriasArmenta2020} a first counterexample to Statement \ref{Iva_axiom} was constructed on 13 vertices and a smaller construction on 11 vertices was found by Ghosh and Ghosh in \cite{Ghosh2021}. Here the authors also challenged the reader to find yet a smaller construction, and we have the following answer to their question.
%
\begin{proposition}\label{thm:axiom}
The two graphs with integer labeled vertices depicted in Figure \ref{fig:notdis} (both on 8 vertices) contradict Axiom 3.4 from \cite{Ivashchenko1994a}, and are the smallest graphs (in terms of vertices) with this property. There are 133 graphs on $9$ vertices that contradict the Axiom.
\end{proposition}

\begin{proof}
To obtain the graphs depicted in Figure \ref{fig:notdis}, we first considered all isomorphism types of connected graphs on at most $9$ vertices (again using \cite{McKay}). We then applied Algorithm \ref{SIC} to determine which of these graphs belong to the family $\SIC$ (note that by Theorem \ref{homologia-trivial-nsvic} this is equivalent to checking containment in $\II$). We then applied Algorithm \ref{CheckAxiom} to determine which graphs fail Ivashchenko's axiom.

We can also explicitly check that each graph  contradicts the Statement \ref{Iva_axiom}.  For this note that in both graphs the vertex $6$ is not adjacent to vertex $1$, and yet there does not exist $u \in G$ that is not adjacent to $6$ and such that $N(6,u)$ is contractible. Indeed, in each graph the non-neighbors of vertex $6$ are $1$ and $4$, and yet 
$N(6,1)=\{7,3,2,5\}$, $N(6,4)=\{7,8,3,5\}$, which are both disconnected and hence not $\II$-contractible. 

Using our software we verified that there are no graphs on 7 vertices or less that are in $\SIC$ and which contradict Statement \ref{Iva_axiom}, and hence the two labeled graphs in Figure \ref{fig:notdis} are indeed the minimal counterexamples.  In addition we found $133$ graphs on nine vertices that contradict the statement.  The software used to verify these claims, along with the graphs themselves, are available in the repository \cite{gcs}.
\end{proof}

\section{Further discussion and open questions}\label{sec:further}

We end with some open questions. Some of these have been mentioned above but we collect them here for convenience. Recall from Section \ref{sec:notstrong} that we are interested in graphs that are $\II$-contractible but not strong $\II$-contractible.  Our conjectures from that section were as follows.

\newtheorem*{conj:dunce}{Conjecture \ref{conj:dunce}}
\begin{conj:dunce}
The smallest graph (in terms of vertices) with trivial homology that it is not strong $\II$-contractible is shown in Figure \ref{DunceA}.
\end{conj:dunce}

\newtheorem*{conj:dunce2}{Conjecture \ref{conj:dunce2}}
\begin{conj:dunce2}
The graph depicted in Figure \ref{DunceB} is the smallest graph $G$ (in terms of vertices) that has the following properties.
\begin{enumerate}
    \item $G$ is strong $\II$-contractible;
    \item There exists and edge $\{v,w\}$ in $G$ such that $N({v,w})\in \SIC$ and such that $G \backslash \{v,w\}$ is not strong $\II$-contractible.
    \end{enumerate}
\end{conj:dunce2}

Our next collection of open questions address the connection between the classes of graphs $\SVIC$, $\SIC$, and $\II$, and topological properties of their clique complexes.  First recall from Theorem \ref{sicne} and Proposition \ref{kd-kc} that a graph $G$ is strong vertex $\II$-contractible if and only if $\Delta(G)$ is non-evasive.

As for the other classes of graphs, recall also that if $G$ is strong $\II$-contractible then $\Delta(G)$ is collapsible \cite{Espinoza2018}. As far as we know the converse is still open.  The statement was first formulated as a conjecture by the last three authors in \cite{Espinoza2018}, where some computational evidence was also discussed.

\begin{conjecture}
Suppose $G$ is a graph such that its clique complex $\Delta(G)$ is collapsible.  Then $G$ is strong $\II$-contractible.
\end{conjecture}

Finally recall from Theorem \ref{thm:contractible} that if $G$ is $\II$-contractible then $\Delta(G)$ is contractible.  Again the converse is still open; that is, if $\Delta(G)$ is a contractible flag complex is it true that $G$ is $\II$-contractible?  It might be surprising if the class of $\II$-contractible graphs happens to coincide with the class of contractible flag simplicial complexes.  Indeed, as the barycentric subdivision of any simplicial complex is flag, classifying contractible flag complexes should be as hard as classifying contractible simplicial complexes.  Hence we ask the following.

\begin{question}
Does there exists a graph $G$ that is not $\II$-contractible but such that $\Delta(G)$ is contractible?
\end{question}

We have seen that the contractible transformations from Definition \ref{IF}, when applied to a graph $G$, preserve the homotopy type of the underlying clique complex $\Delta(G)$.  Hence if one deletes contractible vertices and edges of $G$ to obtain a single vertex, we obtain a certificate that $\Delta(G)$ is contractible.  On the other hand our examples show that there exist graphs $H$ such that $\Delta(H)$ is contractible, but where a greedy collapsing will not result in a vertex.

Our Algorithm \ref{SIVC} checks if a graph $G$ is strong vertex $\II$-contractible.  As a greedy algorithm it is quite straight-forward, and we have seen that it is sufficient to check for $\II$-contractibility of any graph on less than $11$ vertices.  For larger graphs one could perhaps modify the algorithm, and for this it would be useful to have other examples of graphs that are in $\SVIC$ but where the order of vertex deletion matters.  Similarly, Algorithm \ref{SIC} can fail to detect $\II$-contractibility of a graph on 15 or more vertices, since in general one must ``anti-collapse'' to determine if a graph is in $\II$.

The notions of contractible transformations on a graph discussed here also have potential applications to the computation of \emph{persistent homology}.  In this context we are given a filtration of topological spaces, each of which is the Rips complex of some graph (typically the intersection graph of some cover of a point cloud).  One is interested in computing the homology classes that `persist' in this filtration, that survive via the induced map on homology.  As the contractible transformations on a graph $G$ are functorial and preserve the homotopy type of $\Delta(G)$ one can speed up the process of computing persistent homology by first applying the collapsing transformations to each graph in the sequence.

\begin{question}
Can one apply Algorithm \ref{SIC} to speed up the calculation of persistent homology for Rips complexes of graphs?
\end{question}

This question was also addressed in \cite{Espinoza2018}, and animations of an example of this process is available at \cite{gcs}. A similar approach was utilized for the case of strong collapses (which correspond to 0-dismantlable vertices in our context) by Boissonnat and Pritam in \cite{BoiPri}.  Since 0-dismantlable vertices are a special case of $\SVIC$-contractible vertices, our approach will a priori lead to smaller graphs in the sequence.

\section*{Acknowledgments}
We wish to thank Brendan D. McKay, who provided the database of graphs on 10 and 11 vertices \cite{McKay}, as well as Frank Lutz for helpful discussions.  We are especially grateful to an anonymous referee who provided extensive comments and corrections that helped to substantially improve the paper.

\bibliographystyle{siam}
\bibliography{references}

\end{document}